\documentclass[11pt]{amsart}
\usepackage{amsmath, amssymb,graphicx,mathrsfs,enumerate}
\usepackage[all]{xy}
\usepackage{hyperref}
\usepackage{amsthm}
\usepackage{amssymb}
\usepackage{latexsym}
\usepackage{longtable}
\usepackage{epsfig}
\usepackage{amsmath}
\usepackage{hhline}

\newtheorem{thm}{Theorem}

\newtheorem{prop}[thm]{Proposition}

\newcommand{\Center}{\mathbf{Z}}
\newcommand{\Centralizer}{\mathbf{C}}
\newcommand{\Normalizer}{\mathbf{N}}

\newcommand{\Fitting}{\mathbf{F}}
\newcommand{\Layer}{\mathbf{E}}
\newcommand{\Fit}{\mathbf{F}^{*}}

\begin{document}

\title[conjugacy classes of $\pi$-elements in finite groups]
{On the number of conjugacy classes\\ of $\pi$-elements in finite
groups}

\author{Attila Mar\'oti} \address{Alfr\'ed R\'enyi Institute of
  Mathematics, Re\'altanoda utca 13-15, H-1053, Budapest, Hungary \and Fachbereich
Mathematik, Technische Universit\"{a}t Kaiserslautern, Postfach
3049, 67653 Kaiserslautern, Germany}
\email{maroti.attila@renyi.mta.hu \and maroti@mathematik.uni-kl.de}

\author{Hung Ngoc Nguyen}
\address{Department of Mathematics, The University of Akron, Akron,
Ohio 44325, USA} \email{hungnguyen@uakron.edu}

\thanks{The research of the first author was supported by a Marie Curie
International Reintegration Grant within the 7th European
Community Framework Programme, by the J\'anos Bolyai Research
Scholarship of the Hungarian Academy of Sciences, by an
Alexander von Humboldt Fellowship for Experienced Researchers,
by OTKA K84233, and by the MTA RAMKI Lend\"ulet Cryptography Research Group.}


\date{\today}

\begin{abstract}
Let $G$ be a finite group and $\pi$ be a set of primes. Put
$d_{\pi}(G) = k_{\pi}(G)/|G|_{\pi}$ where $k_{\pi}(G)$ is the number
of conjugacy classes of $\pi$-elements in $G$ and $|G|_{\pi}$ is the
$\pi$-part of the order of $G$. In this paper we initiate the study
of this invariant by showing that if $d_{\pi}(G) > 5/8$ then $G$
possesses an abelian Hall $\pi$-subgroup, all Hall $\pi$-subgroups
of $G$ are conjugate, and every $\pi$-subgroup of $G$ lies in some
Hall $\pi$-subgroup of $G$. Furthermore we have $d_{\pi}(G) = 1$ or
$d_{\pi}(G) = 2/3$. This extends and generalizes a result of
W.~H.~Gustafson.
\end{abstract}
\maketitle

\section{Introduction}

For a finite group $G$ let $d(G)$ be the probability that two
elements of $G$ commute. It is easy to see that $d(G) = k(G)/|G|$
where $k(G)$ denotes the number of conjugacy classes of $G$.
Several authors have studied this invariant under the name of
commutativity degree~\cite{Lescot,Erfanian-Rezaei} or commuting
probability \cite{Gustafson,Rusin,GuralnickRobinson}.

Let $\pi(G)$ be the set of prime divisors of the order of $G$ and
$\pi$ a non-empty set of primes. Furthermore, let $k_{\pi}(G)$ be
the number of conjugacy classes of $\pi$-elements in $G$ and
${|G|}_{\pi}$ the $\pi$-part of the order of $G$. Since $d(G)$
encodes a lot of structural information of $G$, it is expected
that $d_{\pi}(G):= k_{\pi}(G)/{|G|}_{\pi}$ also provides some
information on the $\pi$-local structure of $G$.

Our first observation is that $d(G) \leq d_{\pi}(G) \leq d_{\mu}(G)$
whenever $\mu$ is a subset of $\pi$, see part~(1) of
Proposition~\ref{p1}. In particular, if $\mu$ consists of a single
prime, then $d_{\pi}(G) \leq d_{\mu}(G) \leq 1$ by Sylow's theorems.
In fact, we have $d_{\pi}(G) \leq d(P)$ where $P$ is any Sylow
$p$-subgroup of $G$ for any prime $p$ in $\pi$. From this and a
result of P.\,M.~Neumann~\cite{Neumann} it follows that if
$d_{\pi}(G)$ is bounded from below by a positive constant then $P$
is bounded by abelian by bounded; that is, $P$ is `almost' abelian
for every $p \in \pi$. Furthermore, by the same reason, if $G$ is
$\pi$-solvable and $d_{\pi}(G)$ is bounded from below by a positive
constant then every Hall $\pi$-subgroup of $G$ is bounded by abelian
by bounded.

One of the goals of this work is to impose an explicit lower bound for
$d_{\pi}(G)$ in order to ensure the existence of an abelian Hall
$\pi$-subgroup in $G$.

\begin{thm}
\label{main} Let $\pi$ be a set of primes and $G$ a finite group
with $d_{\pi}(G) > 5/8$. Then $G$ contains an abelian Hall
$\pi$-subgroup, all Hall $\pi$-subgroups of $G$ are conjugate
in $G$, and every $\pi$-subgroup of $G$ lies in some Hall $\pi$-subgroup of $G$.
Furthermore $d_{\pi}(G) = 1$ or $d_{\pi}(G) = 2/3$.
\end{thm}

This theorem can be viewed as a local version and extension of
Gustafson's result~\cite{Gustafson} stating that if $d(G) =
d_{\pi(G)}(G) > 5/8$ then $G$ is abelian. We note that the bound
$5/8$ in the theorem is tight since if $G$ is the direct product of
a group of odd order and the dihedral group $D_{8}$ then $d_{ 2 }(G)
= 5/8$. (Here and in what follows, for a prime $p$ we write
$k_{p}(G)$ and $d_{p}(G)$ in place of $k_{\{ p \}}(G)$ and $d_{\{
p\}}(G)$ respectively.) Also, from the condition $d_{\pi}(G) > 5/8$
it does not follow that $G$ is $\pi$-solvable. For if $G$ is a
non-abelian simple group with a Sylow $3$-subgroup of order $3$ then
$G$ is not $3$-solvable but $d_{3}(G) = 2/3$ by
Proposition~\ref{Burnside}.

Our next goal is to describe groups $G$ with $d_{\pi}(G) = 1$ or
$d_{\pi}(G) = 2/3$. By Proposition~\ref{p3} we see that $d_{\pi}(G)
= 1$ if and only if $G$ has a normal $\pi$-complement and an abelian Hall $\pi$-subgroup. On the
other hand, by Theorem \ref{main} and Propositions~\ref{p1}
and~\ref{p2} we see that $d_{\pi}(G) = 2/3$ if and only if $3 \in
\pi$, $2 \not\in \pi$, $d_{3}(G) = 2/3$ and $d_{\pi \setminus
\{3\}}(G) = 1$. In the next theorem we describe groups $G$ with
$d_{3}(G) = 2/3$. By Propositions~\ref{FG},~\ref{p3} and
Theorem~\ref{main}, for this we may assume that $O_{3'}(G) = 1$.

\begin{thm}
\label{main2} Let $G$ be a finite group with $d_{3}(G) = 2/3$ and
$O_{3'}(G) = 1$. Let $P$ be a Sylow $3$-subgroup in $G$. Then $P$ is
abelian, $\Normalizer_{G}(P)/\Centralizer_{G}(P)$ has order $2$,
$[P,\Normalizer_{G}(P)]$ has order $3$, and one of the following
holds.
\begin{enumerate}
\item[\textup{(1)}] $P$ is a self-centralizing normal subgroup in $G$; or

\item[\textup{(2)}] $G = A \times B$ where $A$ is an almost simple finite group with a Sylow $3$-subgroup of order $3$ contained in the socle of $A$ and $B$ is an abelian $3$-group.
\end{enumerate}
\end{thm}

Theorem \ref{main} is independent of the classification theorem of
finite simple groups, however Theorem~\ref{main2} depends on the
fact that if $S$ is a non-abelian finite simple group with a Sylow
$3$-subgroup of order $3$ then the size of the outer automorphism
group of $S$ is not divisible by $3$.

\section{Proof of Theorem \ref{main}}

The starting point of our investigations is the following result
which was communicated to one of us in 2001.

\begin{prop}[Robinson; \cite{Robinson}]
\label{Robinson} Let $\pi = \{ p_{1}, \ldots , p_{t} \}$ be a
subset of $\pi(G)$ for a finite group $G$. Then there exists a
$p_{i}$-subgroup $Q_{i}$ of $G$ for each $i$ with $1 \leq i \leq
t$ so that $k_{\pi}(G) \leq \prod_{i=1}^{t} k(Q_{i})$.
\end{prop}

Hall's theorem about solvable groups extends to $\pi$-solvable
groups (assuming the Odd Order Theorem \cite{Feit-Thompson}).

\begin{prop}[Hall; Theorems 6.4.5 and 6.4.6 of \cite{KS}]
\label{Hall} A $\pi$-solvable group $G$ contains a Hall
$\pi$-subgroup, all Hall $\pi$-subgroups of $G$ are conjugate, and every $\pi$-subgroup in $G$ lies in some Hall $\pi$-subgroup.
\end{prop}

The proof of Proposition~\ref{Robinson} can be used to establish the
following claims.

\begin{prop}
\label{p1} Let $G$ be a finite group and let $\mu \subseteq \pi$
be two non-empty sets of primes.

\begin{enumerate}
\item[\textup{(1)}] Then $d(G) \leq d_{\pi}(G) \leq d_{\mu}(G)\leq
1$. Moreover if $\pi$ is the disjoint union $\mu \cup \{ p \}$
then $k_{\pi}(G) \leq k_{\mu}(G) k_{p}(N)$ for some subgroup $N$
of $G$.

\item[\textup{(2)}] Suppose that $\pi$ is the disjoint union $\mu
\cup \{ p \}$ and that $G$ is $\mu$-solvable with $d_{\mu}(G) =
1$. Then $d_{\pi}(G) = (1/|H|) \sum_{h \in H}
k_{p}(\Centralizer_{G}(h))/{|G|}_{p}$ where $H$ is an abelian Hall
$\mu$-subgroup of $G$.

\item[\textup{(3)}] If $G$ contains an abelian Hall $\pi$-subgroup
then $\prod_{p \in \pi} d_{p}(G) \leq d_{\pi}(G)$.
\end{enumerate}
\end{prop}

\begin{proof}
Assume that $\pi$ is the disjoint union of $\mu$ and $\{ p \}$.
Put $k= k_{\mu}(G)$ and let $x_{1}, \ldots , x_{k}$ be
representatives of the $G$-conjugacy classes of $\mu$-elements of
$G$. For each $1 \leq i \leq k$ let $y_{i,1}, \ldots , y_{i,m(i)}$
be representatives of the $m(i) = k_{p}(\Centralizer_{G}(x_{i}))$
conjugacy classes of $p$-elements inside
$\Centralizer_{G}(x_{i})$.

We claim that any $\pi$-element $z$ of $G$ is conjugate to
$x_{i}y_{i,j}$ for some $i$ and $j$. Write $z = xy$ where $x$ is
the $\mu$-part of $z$ and $y$ is the $p$-part of $z$. By
conjugating by a suitable element of $G$ if necessary, we may
assume that $x = x_i$ for some $i$. But then $y$ lies inside
$\Centralizer_{G}(x_{i})$ and therefore is conjugate in
$\Centralizer_{G}(x_{i})$ to some $y_{i,j}$. This proves the
claim. It is also clear that the elements $x_{i}y_{i,j}$ are
pairwise non-conjugate. Thus \[k_{\pi}(G) =
\sum_{i=1}^{k_{\mu}(G)} k_{p}(\Centralizer_{G}(x_{i})).\]

Let $N$ be a subgroup of $G$ satisfying $k_{p}(N) = \max_{1 \leq i
\leq k} k_{p}(\Centralizer_{G}(x_{i}))$. Then $k_{\pi}(G) \leq
k_{\mu}(G) k_{p}(N)$ which gives the second statement of part~(1).
The first statement of part~(1) readily follows.

Suppose now that $G$ is $\mu$-solvable and that $d_{\mu}(G)=1$.
Then, by Proposition~\ref{Hall}, $\{ x_{1}, \ldots , x_{k} \}$ can
be taken to be a Hall $\mu$-subgroup $H$ of $G$. Thus $k_{\pi}(G) =
\sum_{h \in H} k_{p}(\Centralizer_{G}(h))$. After dividing both
sides of this equality by ${|G|}_{\pi}$ we obtain part~(2).

Finally suppose that $G$ contains an abelian Hall $\pi$-subgroup
$H = \prod_{p \in \pi} H_{p}$ where $H_{p}$ is a Sylow
$p$-subgroup of $G$. For $p \in \pi$ let $x_{p,1}, \ldots ,
x_{p,k_{p}(G)}$ be representatives in $H_{p}$ of the $G$-conjugacy
classes of $p$-elements in $G$. It is easy to see that the
$\pi$-elements $\prod_{p \in \pi} x_{p,i_{p}}$ and $\prod_{p \in
\pi} x_{p,j_{p}}$ are conjugate in $G$ if and only if $i_{p} =
j_{p}$ for all $p \in \pi$. This gives $\prod_{p \in \pi} k_{p}(G)
\leq k_{\pi}(G)$, from which part~(3) readily follows.
\end{proof}

Another key tool in our proof of Theorem~\ref{main} is the
following.

\begin{prop}[Burnside's $p$-complement theorem; Theorem 7.2.1 of \cite{KS}]
\label{Burnside} If a finite group $G$ contains a Sylow
$p$-subgroup $P$ with $\Centralizer_{G}(P) = \Normalizer_{G}(P)$, then $G$ contains a normal $p$-complement.
\end{prop}

Although the proof of the next proposition does not require the full
strength of Proposition~\ref{Burnside} (as kindly pointed out to us
by G.\,R.~Robinson), we will give a proof using it (since we will
need Proposition~\ref{Burnside} later anyway).

\begin{prop}
\label{p3} Let $G$ be a finite group and $\pi$ a non-empty set of
primes. Then we have the following.
\begin{enumerate}
\item[\textup{(1)}] If $d_{p}(G) = 1$ for every $p \in \pi$ then
$G$ contains a normal $\pi$-complement and $G$ is $\pi$-solvable.

\item[\textup{(2)}] $d_{\pi}(G) = 1$ if and only if $G$ has a normal $\pi$-complement and an abelian Hall $\pi$-subgroup.
\end{enumerate}
\end{prop}

\begin{proof}
We claim that if $d_{p}(G) = 1$ for $p \in \pi$ then $G$ contains
a normal $p$-complement. Notice that once this claim is
established, part~(1) follows just by applying this fact
repeatedly for the primes in $\pi \cap \pi(G)$. But this claim
follows directly from Proposition \ref{Burnside}.

The `if' direction of part~(2) is clear by Proposition \ref{Hall},
while the `only if' direction of part~(2) follows from part~(1) of
Proposition~\ref{p1}, part~(1), and Proposition \ref{Hall}.
\end{proof}

As we have already described groups $G$ with $d_{\pi}(G) = 1$, we
now consider groups $G$ with $d_{\pi}(G) < 1$. In doing so our
first tool is a result we mentioned before.

\begin{prop}[Gustafson; \cite{Gustafson}]
\label{G} Let $G$ be a finite group with $d(G) > 5/8$. Then $G$ is
abelian.
\end{prop}

An important application of Proposition \ref{G} is the following.

\begin{prop}
\label{p2} Let $G$ be a finite group and $\pi$ a set of primes
such that $d_{\pi}(G) < 1$. Then we have $d_{\pi}(G) \leq 2/3$.
Furthermore if $3 \not\in \pi$ or if $|G|$ is odd, then $d_{\pi}(G) \leq 5/8$.
\end{prop}

\begin{proof}
Assume that $G$ is a finite group with $d_{\pi}(G)
> 5/8$.

Assume first that $d_{p}(G) = 1$ for every $p \in \pi$. Applying
Proposition~\ref{p3}, we have that $G$ is $\pi$-solvable and that
every Hall $\pi$-subgroup of $G$ is abelian by Propositions
\ref{G} and \ref{Hall}. But then part (3) of Proposition \ref{p1}
gives $d_{\pi}(G) = 1$.

Hence we may assume that $d_{p}(G) < 1$ for some $p \in \pi$, and
thus it is sufficient to prove the proposition in the special case
when $\pi = \{ p \}$. Let $P$ be a Sylow $p$-subgroup of $G$. This
must be abelian. Consider the action of the $p'$-group $X =
\Normalizer_{G}(P)/\Centralizer_{G}(P)$ on $P$. Suppose that $X$ has $r$ fixed
points on $P$. Then we must have $r
> |P|/4$.

If $p \geq 5$ then this can only be if $r = |P|$ and thus
$d_{p}(G) = 1$ by Proposition \ref{Burnside}. A contradiction.

If $p=2$ then the only case to consider is when $r = |P|/2$, since
otherwise we may apply Proposition \ref{Burnside} as before.
Suppose that $Q$ is the subgroup of $P$ all of whose elements are
fixed by $X$. We may assume that there exists an element $x$ of
$X$ which acts non-trivially on $P$. Let $\alpha \in P$ be an
element of a non-trivial $\langle x \rangle$-orbit. Then $\alpha$
is mapped to $\alpha \beta$ by $x$ for some $\beta$ in $Q$. This
means that $\alpha$ lies in an $\langle x \rangle$-orbit of length
a non-trivial power of $2$. A contradiction.

Finally let $p=3$. If $d_{3}(G) > 2/3$ then $r > |P|/3$ and this
forces $d_{3}(G) = 1$ by Proposition~\ref{Burnside}. A
contradiction. On the other hand, if $|G|$ is odd then so is $|X|$.
Thus the smallest prime divisor of $|X|$ is at least $5$ and so
$d_{\pi}(G) \leq d_{3}(G) \leq 7/15$.
\end{proof}

We will also need a useful lemma.

\begin{prop}[Fulman and Guralnick; Lemma 2.3 of \cite{FulmanGuralnick}]
\label{FG} Let $G$ be a finite group and $\pi$ a set of primes.
Then $d_{\pi}(G) \leq d_{\pi}(N) d_{\pi}(G/N)$ for any normal
subgroup $N$ of $G$.
\end{prop}

Now we can turn to the proof of Theorem~\ref{main}. Let $G$ be a
finite group and $\pi$ a set of primes with $d_{\pi}(G) > 5/8$.

We first claim that $G$ contains an abelian Hall $\pi$-subgroup.

By Sylow's theorem and Proposition~\ref{G} we may assume that $|\pi|
\geq 2$. Also, by Proposition~\ref{Hall} and Proposition \ref{G},
there is nothing to show when $G$ is a $\pi$-solvable group. So
assume that $S$ is any non-abelian simple composition factor of $G$
which is not $\pi$-solvable. Then $d_{\pi}(G) \leq d_{\pi}(S) < 1$
by Proposition~\ref{FG} and by part~(2) of Proposition~\ref{p3}. By
Proposition~\ref{p2} we may assume that $\pi(S) \cap \pi = \{ 3 \}$
for any such $S$ (for if $\pi(S) \cap \pi$ contains a prime $p$
different from $3$ then $d_{\pi}(G) \leq d_{p}(G) \leq d_{p}(S) \leq
5/8$ by Proposition \ref{p2}). Furthermore if $|S|$ is odd then
$d_{\pi}(S) \leq d_{3}(S) \leq 5/8$ again by Proposition~\ref{p2}.
So $|S|$ must be even and thus we may assume that the set $\mu = \pi
\setminus \{ 3 \}$ consists of primes at least $5$.

Clearly, the group $G$ is a $\mu$-solvable group and so by
Proposition~\ref{Hall} there exists a Hall $\mu$-subgroup in $G$
intersecting every conjugacy class of $\mu$-elements in $G$.
Moreover, as $d_\mu(G)\geq d_\pi(G)>5/8$ by Proposition~\ref{p1}, we
must have $d_{\mu}(G) = 1$ by Proposition~\ref{p2}. Let $H_1$ be a
Hall $\mu$-subgroup of $G$. Then $d_{\pi}(G) = (1/|H_1|) \sum_{h \in
H_1} k_{3}(\Centralizer_{G}(h))/{|G|}_{3}$ by part~(2) of
Proposition~\ref{p1}.

Given $h \in H_1$. Then $k_{3}(\Centralizer_{G}(h))/{|G|}_{3} \leq
1/3$, or $\Centralizer_{G}(h)$ contains a Sylow $3$-subgroup of
$G$, in which case $k_{3}(\Centralizer_{G}(h))/{|G|}_{3} \leq 1$.
We claim that $\Centralizer_{G}(h)$ contains a Sylow $3$-subgroup
of $G$ for more than $|H_1|/4$ of the $h$'s. For otherwise we
would have $$5/8 < d_{\pi}(G) \leq (3/4)(1/3) + (1/4)(1) = 1/2$$
which is a contradiction.

This means that a Sylow $3$-subgroup $P$ of $G$ centralizes more
than $|H_1|/4$ $\mu$-elements in $G$ (from different $G$-orbits).
But $\Centralizer_{G}(P)$ is also a $\mu$-solvable group and so
contains a Hall $\mu$-subgroup $K$. Then we have $|H_1|/4 <
k_{\mu}(\Centralizer_{G}(P)) \leq |K|$ which forces $|K| = |H_1|$
since $\mu$ consists of primes at least $5$. Thus $H:=K \times P$
is an abelian Hall $\pi$-subgroup in $\Centralizer_{G}(P)$ and
also in $G$. This finishes the proof of our claim.

Next we claim that $d_{\pi}(G) = 2/3$ or $d_{\pi}(G) = 1$.

By parts (1) and (3) of Proposition~\ref{p1} we have $\prod_{p \in
\pi} d_{p}(G) \leq d_{\pi}(G) \leq d_{q}(G)$ for any prime $q$ in
$\pi$. By Proposition \ref{p2}, we may take $q=3$ and we also see
that $d_{p}(G) = 1$ for any prime $p \not= 3$ in $\pi$. This forces
$d_{\pi}(G) = d_{3}(G)$. By Burnside's Lemma (see
\cite[Theorem~7.1.5]{KS}), we know that $d_{3}(G) =
d_{3}(\Normalizer_{G}(P))$ which must be $2/3$ or $1$ (if larger
than $5/8$).

Note that the above arguments show that if $d_{\pi}(G) = 2/3$ then
$3 \in \pi$, $d_{3}(G) = 2/3$, and also that $d_{\mu}(G) = 1$ where
$\mu = \pi \setminus \{ 3 \}$.

Finally, to finish the proof of Theorem \ref{main}, we show that if $L$ is a $\pi$-subgroup of $G$ then
it is conjugate to some subgroup of $H$, where $H$ is a fixed abelian Hall $\pi$-subgroups of $G$.

If $d_{\pi}(G) = 1$ then there is nothing to show by
Propositions~\ref{p3} and~\ref{Hall}. (Or in general we may assume
that $G$ is not $\pi$-solvable.) Thus we may assume that $d_{\pi}(G)
= 2/3$, in particular that $3 \in \pi$ and $d_{\mu}(G) = 1$ where
$\mu = \pi \setminus \{ 3 \}$. (By Sylow's Theorem we may also
assume that $\mu \not= \emptyset$.) Write $H$ in the form $P \times
K$ where $P$ is an abelian Sylow $3$-subgroup of $G$ and $K$ is an
abelian Hall $\mu$-subgroup of $G$.

Since $d_{\mu}(G) = 1$, the group $G$ contains a normal
$\mu$-complement, say $M$, by Proposition~\ref{p3}. So $L$ projects
into $G/M$ with kernel equal to a (normal) Sylow $3$-subgroup
$P_{1}$ of $L$. By the weak version of the Schur-Zassenhaus Theorem
(see \cite[ Theorem~3.3.1]{KS}), $P_{1}$ has an (abelian) complement
$K_{1}$ in $G$. Since $G$ is a $\mu$-solvable group, the subgroup
$K_{1}$ is conjugate to a subgroup of $K$ by Proposition~\ref{Hall}.
Thus, to show that $L$ is conjugate to a subgroup of $H$, we may
assume that $K_{1} \leq K$.

Now consider the $K_{1}$-orbits in $P_{1}$. Since $d_{3}(G) = 2/3$,
these have lengths $1$ or $2$. But there cannot be a $K_{1}$-orbit
of length $2$ in $P_{1}$ since we may assume that $\mu$ consists of
primes at least $5$, as our application of Proposition~\ref{FG}
above shows. This means that $P_{1}$ is central in $L$ and so $L =
P_{1} \times K_{1}$. But then both $P_{1}$ and $P$ lie in
$\Centralizer_{G}(K_{1})$ and so there exists $c \in
\Centralizer_{G}(K_{1})$ with ${P_{1}}^{c} \leq P$. With this same
$c$ we have $L^c \leq H$.

\section{Proof of Theorem \ref{main2}}

In this section we turn to the proof of Theorem \ref{main2}.

Let $G$ be a finite group with $d_{3}(G) = 2/3$ and $O_{3'}(G) = 1$.
Let $P$ be a Sylow $3$-subgroup in $G$. Then $P$ is abelian by
Proposition \ref{G}. By Burnside's Lemma (see Theorem 7.1.5 in
\cite{KS}), we see that $2/3 = d_{3}(G) =
d_{3}(\Normalizer_{G}(P))$. This implies that
$\Normalizer_{G}(P)/\Centralizer_{G}(P)$ has order $2$, the subgroup
$\Normalizer_{G}(P)$ centralizes a subgroup of index $3$ in $P$, and
so $[P,\Normalizer_{G}(P)]$ has order $3$ and $P =
[P,\Normalizer_{G}(P)] \times (P \cap \Center(\Normalizer_{G}(P)))$.

Since $O_{3'}(G) = 1$, the Fitting subgroup $\Fitting(G)$ of $G$ is
a $3$-group. Also, by the Focal Subgroup Theorem, see \cite[pages
165-167]{KS}, we have that $P \cap G'$ has order $3$.

Let $E = \Layer(G)$ be the subgroup generated by the components of
$G$. We claim that $E = 1$ or $E$ is a simple group with a Sylow
$3$-subgroup of order $3$.

Suppose that $E \not= 1$. Then the Sylow $3$-subgroup of $E$ has
order $3$ since $E \leq G'$ and $O_{3'}(G)=1$. Again since
$O_{3'}(G)=1$, the subgroup $\Center(E)$ is a $3$-group. Then, by
the Focal Subgroup Theorem, $\Center(E)$ has order $1$ or $3$. If
$|\Center(E)| = 3$ then $E$ must be the central product of simple
Suzuki groups by \cite{TW}, however a simple Suzuki group cannot
have a Schur multiplier of order divisible by $3$. Thus $\Center(E)
= 1$. So $E$ is a direct product of simple groups and each factor
has order divisible by $3$, whence $E$ is a simple group whose Sylow
$3$-subgroup has order $3$.

So in this case the generalized Fitting subgroup $\Fit(G)$ is $E
\times B$ where $B$ is an abelian $3$-group. Note that $G$
centralizes $B$ (since $B \cap G'=1$) and so $B = \Center(G)$. Thus
$G/B$ embeds in the automorphism group of $E$ (which is almost
simple since $E$ is simple). Inspection of the simple groups with
Sylow $3$-subgroup of order $3$ shows that $G/EB$ has order prime to
$3$. As $H^{2}(E,B)=0$ (since $E$ has a cyclic Sylow $3$-subgroup,
its Schur multiplier has order prime to $3$), the extension $1
\rightarrow B \rightarrow G \rightarrow G/B \rightarrow 1$ splits
and therefore (2) holds.

If $E = 1$, then $\Fit(G) = B$ is a $3$-group and $G/B$ acts
faithfully on $B$, whence $B=P$ and $G/B$ has order $2$. Therefore
(1) holds and the proof is complete.

\section*{Acknowledgement} The authors are grateful to the referee for the careful reading of the manuscript and for supplying them with Theorem~\ref{main2} together with its proof.

\end{document}